\documentclass[11pt]{amsart}

\usepackage[margin=1.22in]{geometry}
\usepackage{amsmath, amssymb, amsthm, graphicx, enumerate, tikz, float, color}
\usepackage{mathtools, comment}
\usepackage{wasysym}
\usepackage[misc]{ifsym}
\usepackage[colorlinks]{hyperref}
\usepackage{orcidlink}

\hypersetup{citecolor=blue}
\usetikzlibrary{matrix,arrows,decorations.pathmorphing}

\newtheorem{theorem}{Theorem}[section]

\newtheorem*{main}{Main Theorem}

\theoremstyle{definition}

\newtheorem{example}[theorem]{Example}

\theoremstyle{remark}
\newtheorem{remark}[theorem]{Remark}

\DeclareMathOperator{\cd}{cd}
\DeclareMathOperator{\Irr}{Irr}
\DeclareMathOperator{\Stab}{Stab}

\numberwithin{equation}{section}

\makeatletter
\@namedef{subjclassname@2020}{%
  \textup{2020} Mathematics Subject Classification}
\makeatother

\begin{document}

\allowdisplaybreaks

\title[Solvable groups whose graphs generalize the bowtie]{Constructing solvable groups whose character degree graphs generalize the bowtie}

\author[J. Laubacher | M.L. Lewis | L. Ravaglia | A. Summers]{Jacob Laubacher \orcidlink{0000-0003-0045-7951} | Mark L. Lewis \orcidlink{0000-0001-9627-6922} | Lorenzo Ravaglia \orcidlink{0009-0000-4891-4854} | Andrew Summers \orcidlink{0009-0005-5561-6309}}
\address{Jacob Laubacher \orcidlink{0000-0003-0045-7951} | Department of Mathematics, Hillsdale College | Hillsdale, Michigan 49242, USA}
\email{\textcolor{magenta}{jlaubacher@hillsdale.edu}}

\address{Mark L. Lewis \orcidlink{0000-0001-9627-6922} | Department of Mathematical Sciences, Kent State University | Kent, Ohio 44242, USA}
\email{\textcolor{magenta}{lewis@math.kent.edu}}

\address{Lorenzo Ravaglia \orcidlink{0009-0000-4891-4854} | Department of Mathematics, Hillsdale College | Hillsdale, Michigan 49242, USA}
\email{\textcolor{magenta}{lravaglia@hillsdale.edu}}

\address{Andrew Summers \orcidlink{0009-0005-5561-6309} | Department of Mathematical Sciences, Kent State University | Kent, Ohio 44242, USA}
\email{\textcolor{magenta}{asumme19@kent.edu}}

\date{\today}


\keywords{Character degrees, solvable groups, families of graphs\\\indent\emph{Corresponding author.} Jacob Laubacher \Letter~\href{mailto:jlaubacher@hillsdale.edu}{jlaubacher@hillsdale.edu} \phone~(517) 607-3285.}

\begin{abstract}
We present here a generalized construction of a finite solvable group whose prime character degree graph has the shape and structure of the bowtie graph. As with the original bowtie, the graphs obtained by this generalized construction, under certain restrictions, cannot be realized by the usual method of taking direct products of smaller graphs. Within the condition of $n=1$, we show how this recovers the original bowtie graph, which has five vertices. We also provide examples and explicit choices of primes which generate graphs with more vertices.
\end{abstract}

\maketitle

\section{Introduction}

Throughout this paper, we will let $G$ denote a finite solvable group. We follow convention and write $\Irr(G)$ for the set of irreducible characters of $G$ and $\cd(G)$ for the set of irreducible character degrees of $G$. With this in mind, one can then draw the corresponding prime character degree graph of $G$, which we denote by $\Delta(G)$. Let $V(\Delta(G))$ denote the vertex set of $\Delta(G)$. There is a vertex $p\in V(\Delta(G))$ if $p$ is a prime number and $p\mid a$ for some $a\in\cd(G)$. Furthermore, there is an edge between two vertices $p$ and $q$ in $V(\Delta(G))$ if $pq\mid b$ for some $b\in\cd(G)$. By construction, $\Delta(G)$ is a simple graph, and we note that the words ``prime" and ``vertex" become synonymous in this context. 

One approach to studying prime character degree graphs is via classifications of graphs with a fixed number of vertices (like \cite{L5}, for example). In such classifications, the main question one asks is: given a graph $\Gamma$, does there exist a finite solvable group $G$ such that $\Delta(G)=\Gamma$? When such a group $G$ exists, the graph $\Gamma$ is said to ``occur" and is referred to as an ``occurring graph." Otherwise, the graph is said to be ``non-occurring," or in some cases, remains unclassified.

These classifications often rely on a wide range of approaches, from landmark results like P\'alfy's condition in \cite{P} (or its generalized version as seen in \cite{A}), to structural results such as the main theorems of \cite{LM}, or even investigations of certain families of related graphs as in \cite{BL}. For a given graph $\Gamma$, it is often much easier and more common to answer ``no" to the question of occurrence, meaning that there is no solvable group $G$ such that $\Delta(G)=\Gamma$. In fact, it is a relative rarity for a graph to occur as the prime character degree graph of some solvable group, as shown for disconnected graphs in particular in \cite{LSDisc}, as well as in the classifications by number of vertices (see \cite{H}, \cite{Z}, \cite{L5}, \cite{BLL}, \cite{LMS}, \cite{LS}, and \cite{Biss}). We highlight the aggregate of these investigations in Figure \ref{figTable}. Due to this rarity, methods which construct solvable groups that achieve certain graphs are quite valuable.

\begin{figure}[htb]
    \centering
    \resizebox{\textwidth}{!}{
    \begin{tabular}{c|c|c|c|c}
        Number of & Total Number & Number of Graphs & Number of Graphs & Number of Graphs\\
        Vertices & of Graphs & That Occur & That Do Not Occur & Unclassified\\
        \hline
        1 & 1 & 1 & 0 & 0 \\
        2 & 2 & 2 & 0 & 0 \\
        3 & 4 & 3 & 1 & 0 \\
        4 & 11 & 5 & 6 & 0 \\
        5 & 34 & 9 & 24 & 1 \\
        6 & 156 & 15 & 132 & 9 \\
        7 & 1044 & 24 & 976 & 44 \\
        8 & 12346 & 39 & 12103 & 204\\
    \end{tabular}
    }
    \caption{Number of occurring, non-occurring, and unclassified graphs}
    \label{figTable}
\end{figure}

One example of such a construction can be found in \cite{L3}, where Lewis constructs the first instance of a solvable group whose prime character degree graph has a diameter of three. Afterwards, in her dissertation (see \cite{D}), Dugan generalized the aforementioned construction, allowing for other graphs of diameter three to be classified as ``occurring." The graph that Lewis constructed has six vertices and is therefore formally classified in \cite{BLL} while the generalized construction from Dugan has been employed to classify diameter three graphs with seven vertices (see \cite{LMS}) and eight vertices (see \cite{LS}), which nicely showcases the use and importance of such a generalized construction. 

Likely the most common method for showing a given graph occurs is via direct products of smaller occurring character degree graphs. Let $G$ and $H$ be finite solvable groups with character degree graphs $\Delta(G)$ and $\Delta(H)$, respectively. The direct product $G\times H$ is also a finite solvable group, and, following theorem 4.21 of \cite{I}, $\Delta(G\times H)$ is relatively easy to determine. The vertex set $V(\Delta(G\times H))$ is simply the union $V(\Delta(G))\cup V(\Delta(H))$, and an edge is drawn between two vertices $p$ and $q$ if and only if one of the following conditions is met:
\begin{enumerate}
    \item $pq$ is an edge of $\Delta(G)$;
    \item $pq$ is an edge of $\Delta(H)$;
    \item $p\in V(\Delta(G))$ and $q\in V(\Delta(H))$;
    \item $p\in V(\Delta(H))$ and $q\in V(\Delta(G))$.
\end{enumerate}
With this in mind, a good number of graphs can be constructed by taking direct  products of occurring character degree graphs with fewer vertices. In some situations, it can even be determined from $\Delta(G)$ that the underlying group $G$ must be a direct product of subgroups (see the main results of \cite{LM}). This direct product technique does have its limitations, however, and it is quite common to come across graphs which cannot be constructed in this way, again highlighting the importance of alternative methods.

In this paper, our goal is to build off of the construction of the bowtie graph in example 7.1 of \cite{L5}, where the solvable group that is constructed models those seen in \cite{I2}. By generalizing this construction, it will allow us to construct solvable groups whose prime character degree graphs specifically cannot be represented as direct products in the sense of the previous paragraph. We formulate our result as follows:

\begin{main}
There exists a finite solvable group $G$ such that its prime character degree graph $\Delta(G)$ is the graph seen in Figure \ref{figpqr}.
\end{main}

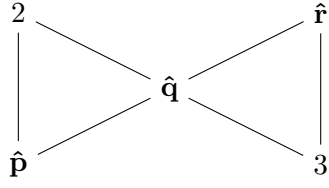
\begin{figure}[htb]
    \centering
    $
\begin{tikzpicture}[scale=2]
\node (2) at (0,1) {$2$};
\node (p) at (0,0) {$\mathbf{\hat{p}}$};
\node (q) at (1,.5) {$\mathbf{\hat{q}}$};
\node (r) at (2,1) {$\mathbf{\hat{r}}$};
\node (3) at (2,0) {$3$};
\path[font=\small,>=angle 90]
(2) edge node [right] {$ $} (p)
(3) edge node [right] {$ $} (r)
(q) edge node [right] {$ $} (2)
(q) edge node [right] {$ $} (p)
(q) edge node [right] {$ $} (3)
(q) edge node [right] {$ $} (r);
\end{tikzpicture}
    $
    \caption{The generalized bowtie graph $\Delta(G)$}
    \label{figpqr}
\end{figure}
We note here that $\mathbf{\hat{p}}$, $\mathbf{\hat{q}}$, and $\mathbf{\hat{r}}$ need not represent single vertices, and will be discussed much more carefully in the following section.

Finally, we note that some of the work in this paper was completed by the fourth author as a Ph.D. candidate under the supervision of the second author at Kent State University. The contents of this paper may appear as part of the fourth author’s Ph.D. dissertation.

\section{Proof of the Main Theorem}\label{secMain}

Besides a sufficient background in character theory (see \cite{I}), we need only recall the landmark result from P\'alfy pertaining to disconnected graphs:

\begin{theorem}[P\'alfy's inequality from \cite{P2}]
Let $G$ be a solvable group and $\Delta(G)$ its prime character degree graph. Suppose that $\Delta(G)$ is disconnected with two components having size $a$ and $b$, where $a\leq b$. Then $b\geq2^a-1$.
\end{theorem}

We now present a generalization of the construction of the finite solvable group $G$ whose prime character degree graph $\Delta(G)$ is the bowtie graph that Lewis presents in Example 7.1 on page 263 of \cite{L5}. These groups are similar to those constructed in \cite{I2} (see also \cite{DHS}), and we will attempt to follow a similar notational convention.

We start by letting $n\in\mathbb{N}$. We then set $\mathbf{\hat{p}}$ to be a product of $n$ distinct prime numbers, none of which are $2$ and $3$, that satisfy certain conditions. Explicitly, we have that
\begin{equation}\label{phat}
\mathbf{\hat{p}}=p_1p_2\cdots p_n,
\end{equation}
that satisfies the equations
\begin{equation}\label{qrhat}
\mathbf{\hat{q}}=2^{\mathbf{\hat{p}}}-1=q_1q_2\cdots q_m\text{~and~}\mathbf{\hat{r}}=\frac{2^{\mathbf{\hat{p}}}+1}{3}=r_1r_2\cdots r_k,
\end{equation}
where we require that the $q_i$ (for all $1\leq i\leq m$) and the $r_j$ (for all $1\leq j\leq k$) need to be distinct prime numbers that are all different from both $2$ and $3$, as well as from all the primes that make up $\mathbf{\hat{p}}$. Having these all be distinct prime numbers will guarantee that the resulting prime character degree graph will have $n+m+k+2$ vertices. Finally, observe that
\begin{equation}\label{factoring}
2^{2\mathbf{\hat{p}}}-1=(2^{\mathbf{\hat{p}}}-1)(2^{\mathbf{\hat{p}}}+1)=3\mathbf{\hat{q}}\mathbf{\hat{r}}=3q_1q_2\cdots q_mr_1r_2\cdots r_k,
\end{equation}
which will be useful later upon converting the set of character degrees to the corresponding prime character degree graph.

Doing a few computations shows that many sets of primes that make up $\mathbf{\hat{p}}$ from \eqref{phat} seem to satisfy \eqref{qrhat}, and in Section \ref{secEx}, we will present several examples. However, one can find sets of primes that do not satisfy \eqref{qrhat}. It would be an interesting question in Number Theory to ask if there are infinitely many such sets of primes.

\begin{main}
There exists a finite solvable group $G$ such that its prime character degree graph $\Delta(G)$ is the graph seen in Figure \ref{figpqr}, where $\mathbf{\hat{p}}$, $\mathbf{\hat{q}}$, and $\mathbf{\hat{r}}$ are all defined as above.
\end{main}
\begin{proof}
We begin with a hefty amount of set-up. So, with the above in place, we start by taking $\mathbb{F}$ to be the field of order $2^{\mathbf{\hat{p}}}$. Consequently, we then denote $\mathbb{E}$ to be the field extension of $\mathbb{F}$ whose size is $2^{2\mathbf{\hat{p}}}$. As in \cite{L5}, we note that $\mathbb{E}$ has the field automorphism $\sigma$ given by $\sigma(\alpha)=\alpha^{2^{\mathbf{\hat{p}}}}$ for all $\alpha\in\mathbb{E}$, and furthermore we know $\mathbb{F}$ is fixed under $\sigma$ and that $\sigma^2=1$.

Letting $\mathcal{R}$ be the skew polynomial ring with coefficients in $\mathbb{E}$ and the indeterminate $X$, we then define $X\alpha=\alpha^\sigma X$ as the operation of multiplication between an arbitrary element of the extension field $\mathbb{E}$ and the indeterminate $X$. Notice that $\mathcal{R}X^3=X^3\mathcal{R}$, and so $\mathcal{R}X^3$ is an ideal of $\mathcal{R}$. We can then define $R$ to be the corresponding quotient ring: $R=\mathcal{R}/\mathcal{R}X^3$. Moreover, we define $x$ to be the image of the indeterminate $X$ in $R$, and we see that $J(R)=Rx$, where $J(R)$ denotes the Jacobson radical of $R$. Next, define $Q=1+Rx=\{1+\alpha x+\beta x^2~|~\alpha,\beta\in\mathbb{E}\}$, and it is easy to verify that $Q$ is a group of order $2^{4\mathbf{\hat{p}}}$. One can then show that the commutator
$$
[1+\alpha_1x+\beta_1x^2,1+\alpha_2x+\beta_2x^2]=1+(\alpha_1\alpha_2^\sigma+\alpha_1^\sigma\alpha_2)x^2.
$$
As a result, one gleans $Q'=\{1+\delta x^2~|~\delta\in\mathbb{F}\}$.

Next, denote $C$ to be the multiplicative group of $\mathbb{E}$. This immediately implies that $C$ is cyclic, and that since $|\mathbb{E}|=2^{2\mathbf{\hat{p}}}$, then $|C|=2^{2\mathbf{\hat{p}}}-1$. Following \eqref{factoring}, one can instead write $|C|=3\mathbf{\hat{q}}\mathbf{\hat{r}}$. Taking $\mathcal{G}$ to be the Galois group of $\mathbb{E}$ over its base field, we can conclude that $\mathcal{G}$ is cyclic and that $|\mathcal{G}|=2\mathbf{\hat{p}}$. Beyond $\mathcal{G}$ acting on $C$ in the natural way, one can also define $C$ acting on $Q$ via $(1+\alpha x+\beta x^2)^c=1+\alpha cx+\beta c^{2^{\mathbf{\hat{p}}}+1}x^2$ for all $c\in C$ and $\mathcal{G}$ acting on $R$ by $(1+\alpha x+\beta x^2)^g=1+\alpha^gx+\beta^gx^2$ for all $g\in\mathcal{G}$. As some of the final steps of bookkeeping, we define $C_1$ such that $C_1\leq C$ and $|C_1|=2^{\mathbf{\hat{p}}}+1=3\mathbf{\hat{r}}$ along with $C_2$ such that $C_2\leq C$ with $|C_2|=2^{\mathbf{\hat{p}}}-1=\mathbf{\hat{q}}$. We note that $C_2$ acts Frobeniusly on $Q$. Moreover, by denoting $Z=C_Q(C_1)$, we see that $Z=\{1+\gamma x^2~|~\gamma\in\mathbb{E}\}$. We can then apply Fitting's lemma to get that $Q/Q'=[Q,C_1]Q'/Q'\times Z/Q'$. For ease of notation, set $P=[Q,C_1]Q'$. Observe that
$$
|Q:P|=|Q:[Q,C_1]Q'|=|Z:Q'|=2^{\mathbf{\hat{p}}}.
$$
By above, since we know $|Q|=2^{4\mathbf{\hat{p}}}$, and since $|Q:P|=2^{\mathbf{\hat{p}}}$, then we can conclude that $|P|=2^{3\mathbf{\hat{p}}}$. Finally, one can show that $Z(P)=P'=Q'$ and it is not difficult to see that both $C$ and $\mathcal{G}$ normalize $P$.

We are now ready to define the finite solvable group $G$. Set $G=P\rtimes C\rtimes\mathcal{G}$. Our goal now is to compute the set of character degrees for $G$.

As a first step, it is not difficult to show that $\cd(C\rtimes\mathcal{G})=\{d~|~d\in\mathbb{N}\text{~and~}d\text{~divides~}2\mathbf{\hat{p}}\}$. Furthermore, $C$ acts transitively on the nonprincipal characters in $\Irr(P/P')$, and so if $\delta\in\Irr(P/P')$ such that $\delta$ is nonprincipal, then one can show that $\delta$ lies in an orbit of size $2^{2\mathbf{\hat{p}}}-1$. Supposing that $\Stab_G(\delta)=P\rtimes\mathcal{G}$, since $\mathcal{G}$ is cyclic, then we know that the irreducible characters $\Irr(P\rtimes\mathcal{G}|\delta)$ have the extensions of $\delta$. Therefore, $\cd(G|\delta)=\{2^{2\mathbf{\hat{p}}}-1\}$, and again by \eqref{factoring}, we can instead write $\cd(G|\delta)=\{3\mathbf{\hat{q}}\mathbf{\hat{r}}\}$. Hence,
\begin{equation}\label{first}
\cd(G/P')=\{d,3\mathbf{\hat{q}}\mathbf{\hat{r}}~|~d\in\mathbb{N}\text{~and~}d\text{~divides~}2\mathbf{\hat{p}}\}.
\end{equation}
 
Recall that $C_2$ acts Frobeniusly on $Q$, and therefore also on $Q'=P'=Z(P)$ and $Q'\le C_Q(C_1)$. It follows that $P\rtimes C_1$ centralizes $P'$. As $P'=Z(P)$, we see that $P'$ is abelian and therefore the action of $C_2$ on $P'$ is permutation isomorphic to the action of $C_2$ on the elements of $\Irr(P')$. Since $|C_2|=2^{\hat{p}}-1=|\Irr(P')|-1$, it follows that $C_2$ acts transitively on the nonprincipal characters of $\Irr(P')$, and so $\lambda$ lies in an orbit of size $2^{\mathbf{\hat{p}}}-1=\mathbf{\hat{q}}$. Because $C_2$ acts transitively on $\Irr(P')\setminus\{1\}$, the stabilizers of the characters in $\Irr(P')\setminus\{1\}$ in $C_2\rtimes\mathcal{G}$ correspond to the conjugates of $\mathcal{G}$. Thus, without loss of generality, we can choose $1\ne\lambda\in\Irr(P')$ so that the stabilizer of $\lambda$ in $C_2\rtimes\mathcal{G}$ is $\mathcal{G}$. Thus, the stabilizer of $\lambda$ in $G$ is $P\rtimes C_1\rtimes\mathcal{G}$. Observe that $P/P'$ is irreducible under the action of $C_1$, which implies that $P/P'$ is a chief factor in $P\rtimes C_1\rtimes\mathcal{G}$. This implies that $P/P'$ in $P\rtimes C_1\rtimes\mathcal{G}$ and $\lambda\in\Irr(P')$ satisfy the hypotheses of Problem 6.12 in \cite{I}. By the conclusions of that problem, we see that either $\lambda$ induces irreducibly to $P$, $\lambda$ is fully-ramified with respect to $P/P'$, or $\lambda$ extends to $P$. Since $\lambda$ is invariant in $P$, it does not induce to $P$. Next, suppose $\lambda$ extends to $\tilde\lambda$. Then $\tilde\lambda$ would be linear, and since $(\tilde\lambda)_{P'}=\lambda\ne1$, we have a contradiction. Thus, we are forced to conclude that $\lambda$ must be fully-ramified with respect to $P/P'=P/Z(P)$.

Denoting $\hat{\lambda}$ as the irreducible constituent of $\lambda^P$, we get that $\Stab_G(\hat{\lambda})=\Stab_G(\lambda)=P\rtimes C_1\rtimes\mathcal{G}$. Since we know all the Sylow subgroups of $C_1\rtimes\mathcal{G}$ are cyclic, we have that $\hat{\lambda}$ extends to $P\rtimes C_1\rtimes\mathcal{G}$. We now consider the set $\cd(C_1\rtimes\mathcal{G})$. Let $\mathbf{p^*}$ be a positive (not necessarily proper) divisor of $\mathbf{\hat{p}}$. It follows that $S=\langle\sigma^{\mathbf{p^*}}\rangle$ is a subgroup of $\mathcal{G}$ having even order $2\mathbf{\hat{p}}/\mathbf{p^*}$. Note that $\alpha$ is in the fixed field of $S$ exactly when
$$
\sigma^{\mathbf{p^*}}(\alpha)=\alpha\iff\alpha^{2^{\mathbf{p^*}}}=\alpha\iff\alpha^{2^{\mathbf{p^*}}-1}=1\iff o(\alpha)\mid(2^{\mathbf{p^*}}-1).
$$
Noting that $(2^{\mathbf{p^*}}-1)\mid(2^{\mathbf{\hat{p}}}-1)$ and that $(2^{\mathbf{\hat{p}}}+1,2^{\mathbf{\hat{p}}}-1)=(3\mathbf{\hat{r}},\mathbf{\hat{q}})=1$, it follows that $\alpha\in C_2$ and we have that $I_{\mathcal{G}}(\lambda)\neq S$ for every $\lambda\in\Irr(C_1)$. Thus, either $2\mid |\mathcal{G}:I_{\mathcal{G}}(\lambda)|=\lambda^{C_1\rtimes\mathcal{G}}(1)$ or $\lambda^{C_1\rtimes\mathcal{G}}(1)=|\mathcal{G}:I_{\mathcal{G}}(\lambda)|=1$, and therefore $\cd(C_1\rtimes\mathcal{G})=\{1,2\hat{d}~|~\hat{d}\in\mathbb{N}\text{~and~}\hat{d}\text{~divides~}\mathbf{\hat{p}}\}$. By Gallagher's theorem, we then get that $\cd(P\rtimes C_1\rtimes\mathcal{G}|\hat{\lambda})=\cd(P\rtimes C_1\rtimes\mathcal{G}|\lambda)=\{2^{\mathbf{\hat{p}}},2^{\mathbf{\hat{p}}+1}\hat{d}~|~\hat{d}\in\mathbb{N}\text{~and~}\hat{d}\text{~divides~}\mathbf{\hat{p}}\}$. Hence, $\cd(G|\lambda)=\{2^{\mathbf{\hat{p}}}\mathbf{\hat{q}},2^{\mathbf{\hat{p}}+1}\mathbf{\hat{q}}\hat{d}~|~\hat{d}\in\mathbb{N}\text{~and~}\hat{d}\text{~divides~}\mathbf{\hat{p}}\}$. However, since all the nonprincipal characters of $\Irr(P')$ are conjugate, then in particular we have that
\begin{equation}\label{second}
\cd(G|P')=\{2^{\mathbf{\hat{p}}}\mathbf{\hat{q}},2^{\mathbf{\hat{p}}+1}\mathbf{\hat{q}}\hat{d}~|~\hat{d}\in\mathbb{N}\text{~and~}\hat{d}\text{~divides~}\mathbf{\hat{p}}\}.
\end{equation}

Since we know $\cd(G)=\cd(G/P')\cup\cd(G|P')$, then by combining \eqref{first} and \eqref{second}, we get that
\begin{equation}\label{degrees}
\cd(G)=\{d,3\mathbf{\hat{q}}\mathbf{\hat{r}},2^{\mathbf{\hat{p}}}\mathbf{\hat{q}},2^{\mathbf{\hat{p}}+1}\mathbf{\hat{q}}\hat{d}~|~d,\hat{d}\in\mathbb{N},~d\text{~divides~}2\mathbf{\hat{p}},\text{~and~}\hat{d}\text{~divides~}\mathbf{\hat{p}}\}.
\end{equation}

Noting again that all the primes are distinct, and also by condensing vertices because they form complete graphs (as done in \cite{D}), we can then draw the corresponding prime character degree graph $\Delta(G)$ as seen in Figure \ref{figpqr}.
\end{proof}

\section{Examples}\label{secEx}

Here we provide specific examples of primes to show the use of the construction seen in Section \ref{secMain}.

\begin{example}\label{exn1m1}
Consider $n=1$, $m=1$, and $k=1$. Thus we have $\mathbf{\hat{p}}=p_1$, and this then yields $\mathbf{\hat{q}}=2^{p_1}-1=q_1$ and $\mathbf{\hat{r}}=\frac{2^{p_1}+1}{3}=r_1$. Following \eqref{degrees}, we get
$$
\cd(G_1)=\{1,2,p_1,2p_1,3q_1r_1,2^{p_1}q_1,2^{p_1+1}q_1,2^{p_1+1}p_1q_1\}.
$$
This reduces to the construction that Lewis performs in \cite{L5}, which yields the bowtie graph with five vertices shown in Figure \ref{fign1}. Explicitly, taking $p_1=5$ results in $q_1=31$ and $r_1=11$, which is the example of the prime numbers that Lewis provides.
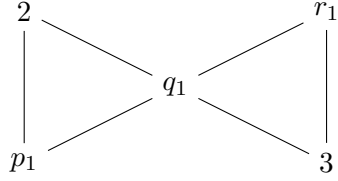
\begin{figure}[htb]
    \centering
    $
\begin{tikzpicture}[scale=2]
\node (2) at (0,1) {$2$};
\node (p) at (0,0) {$p_1$};
\node (q) at (1,.5) {$q_1$};
\node (r) at (2,1) {$r_1$};
\node (3) at (2,0) {$3$};
\path[font=\small,>=angle 90]
(2) edge node [right] {$ $} (p)
(3) edge node [right] {$ $} (r)
(q) edge node [right] {$ $} (2)
(q) edge node [right] {$ $} (p)
(q) edge node [right] {$ $} (3)
(q) edge node [right] {$ $} (r);
\end{tikzpicture}
    $
    \caption{The graph $\Delta(G_1)$ corresponding to $n=1$, $m=1$, and $k=1$}
    \label{fign1}
\end{figure}
\end{example}

We note that it is easy to find more examples for the case of $n=1$, $m=1$, and $k=1$ as above. For instance, one can consider $p_1=7$ (forcing $q_1=127$ and $r_1=43$), or even $p_1=13$ (giving $q_1=8191$ and $r_1=2731$) or $p_1=17$ (with $q_1=131071$ and $r_1=43691$), among others.

\begin{remark}\label{remDP}
Based on $n$, the Zsigmondy prime theorem determines that $m\geq2^n-1$. It turns out that it is advantageous to use the smallest $m$ possible. In doing so, one can always take a direct product with the singleton $K_1$ to increase the number of complete vertices in the middle (the $q$'s). The following example showcases this idea.
\end{remark}

\begin{example}\label{exn1m2}
Consider $n=1$, but $m=2$ and $k=1$. We have $\mathbf{\hat{p}}=p_1$, which then gives $\mathbf{\hat{q}}=2^{p_1}-1=q_1q_2$ and $\mathbf{\hat{r}}=\frac{2^{p_1}+1}{3}=r_1$. Again referencing \eqref{degrees}, we get the set of character degrees as
$$
\cd(G_2)=\{1,2,p_1,2p_1,3q_1q_2r_1,2^{p_1}q_1q_2,2^{p_1+1}q_1q_2,2^{p_1+1}p_1q_1q_2\},
$$
and the corresponding prime character degree graph can be seen in Figure \ref{fign1m2}. This example can be realized by taking $p_1=11$, which then yields $q_1=23$, $q_2=89$, and $r_1=683$.
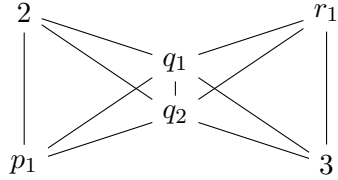
\begin{figure}[htb]
    \centering
    $
\begin{tikzpicture}[scale=2]
\node (2) at (0,1) {$2$};
\node (p) at (0,0) {$p_1$};
\node (q1) at (1,.675) {$q_1$};
\node (q2) at (1,.325) {$q_2$};
\node (r) at (2,1) {$r_1$};
\node (3) at (2,0) {$3$};
\path[font=\small,>=angle 90]
(2) edge node [right] {$ $} (p)
(3) edge node [right] {$ $} (r)
(q1) edge node [right] {$ $} (q2)
(q1) edge node [right] {$ $} (2)
(q1) edge node [right] {$ $} (p)
(q1) edge node [right] {$ $} (3)
(q1) edge node [right] {$ $} (r)
(q2) edge node [right] {$ $} (2)
(q2) edge node [right] {$ $} (p)
(q2) edge node [right] {$ $} (3)
(q2) edge node [right] {$ $} (r);
\end{tikzpicture}
    $
    \caption{The graph $\Delta(G_2)$ corresponding to $n=1$, $m=2$, and $k=1$}
    \label{fign1m2}
\end{figure}
\end{example}

\begin{example}
Following Remark \ref{remDP}, notice how the graph in Figure \ref{fign1m2} can be realized as a direct product. Explicitly, letting $A$ be a finite solvable group whose prime character degree graph $\Delta(A)$ is the singleton $K_1$, we can consider the group $G_1\times A$, where $G_1$ is from Example \ref{exn1m1}. We then see that $\Delta(G_2)$ in Figure \ref{fign1m2} is identical to the graph $\Delta(G_1\times A)$, which is shown in Figure \ref{figG1A}. For reference, we have color-coded the graph so that the blue tracks $\Delta(G_1)$, the red corresponds to $\Delta(A)$, and the black follows the new edges created via the direct product. As is typical in these scenarios, we can guarantee that all primes are distinct and therefore will end up with six total vertices here. The whole graph, of course, is $\Delta(G_1\times A)$.
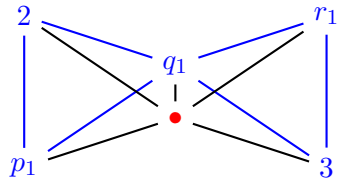
\begin{figure}[htb]
    \centering
    $
\begin{tikzpicture}[scale=2]
\node (2) at (0,1) {\textcolor{blue}{$2$}};
\node (p) at (0,0) {\textcolor{blue}{$p_1$}};
\node (q1) at (1,.675) {\textcolor{blue}{$q_1$}};
\node (q2) at (1,.325) {\textcolor{red}{$\bullet$}};
\node (r) at (2,1) {\textcolor{blue}{$r_1$}};
\node (3) at (2,0) {\textcolor{blue}{$3$}};
\path[font=\small,>=angle 90,blue,thick]
(2) edge node [right] {$ $} (p)
(3) edge node [right] {$ $} (r)
(q1) edge node [right] {$ $} (2)
(q1) edge node [right] {$ $} (p)
(q1) edge node [right] {$ $} (3)
(q1) edge node [right] {$ $} (r);
\path[font=\small,>=angle 90,thick]
(q2) edge node [right] {$ $} (q1)
(q2) edge node [right] {$ $} (2)
(q2) edge node [right] {$ $} (p)
(q2) edge node [right] {$ $} (3)
(q2) edge node [right] {$ $} (r);
\end{tikzpicture}
    $
    \caption{The graph $\Delta(G_1\times A)$}
    \label{figG1A}
\end{figure}
\end{example}

\begin{remark}\label{remPI}
Based on $n$ and $m$, the Zsigmondy prime theorem also gives bounds on $k$, which is $2^{n+1}-m-2\leq k\leq2^{n+1}-3$. The upper bound, in particular, is tied directly to P\'alfy's inequality, thereby ensuring our construction is not redundant. If, for instance, we have $k>2^{n+1}-3$, then we in fact have $k\geq2^{n+1}-2$. In this scenario, the resulting graph is nothing more than a direct product between $K_m$ (the complete graph on $m$ vertices, coming from the primes $q_1,q_2,\ldots,q_m$) and the disconnected graph where one component is $K_{n+1}$ (with primes $2,p_1,p_2\ldots,p_n$)  and the other is $K_{k+1}$ (with the primes $3,r_1,r_2,\ldots,r_k$). Under the aforementioned assumption of $k\geq2^{n+1}-2$, we see that this construction induces $a=n+1$ and $b=k+1$, resulting in
$$
b=k+1\geq2^{n+1}-2+1=2^{n+1}-1=2^a-1,
$$
which, again, is permitted by P\'alfy's inequality. Hence, in order to keep our construction interesting and worthwhile, we demand $k\leq2^{n+1}-3$. We explore the necessity of this bound on $k$ with the next two examples.
\end{remark}

\begin{example}\label{exn1k2}
Consider $n=1$, but $m=1$ and $k=2$. Here we have $\mathbf{\hat{p}}=p_1$, and consequently $\mathbf{\hat{q}}=2^{p_1}-1=q_1$ and $\mathbf{\hat{r}}=\frac{2^{p_1}+1}{3}=r_1r_2$. Following \eqref{degrees}, we get
$$
\cd(G_3)=\{1,2,p_1,2p_1,3q_1r_1r_2,2^{p_1}q_1,2^{p_1+1}q_1,2^{p_1+1}p_1q_1\}.
$$
The prime character degree graph $\Delta(G_3)$ can be seen in Figure \ref{fign1k2}, and we note that this example can be obtained by taking $p_1=107$, which then yields the Mersenne prime $q_1=162259276829213363391578010288127$, as well as the prime numbers $r_1=643$ and $r_2=84115747449047881488635567801$.
\begin{figure}[htb]
    \centering
    $
\begin{tikzpicture}[scale=2]
\node (2) at (0,1) {$2$};
\node (p) at (0,0) {$p_1$};
\node (q) at (1,.5) {$q_1$};
\node (r1) at (1.75,1) {$r_1$};
\node (r2) at (2.25,1) {$r_2$};
\node (3) at (2,0) {$3$};
\path[font=\small,>=angle 90]
(2) edge node [right] {$ $} (p)
(3) edge node [right] {$ $} (r1)
(3) edge node [right] {$ $} (r2)
(r1) edge node [right] {$ $} (r2)
(q) edge node [right] {$ $} (2)
(q) edge node [right] {$ $} (p)
(q) edge node [right] {$ $} (3)
(q) edge node [right] {$ $} (r1)
(q) edge node [right] {$ $} (r2);
\end{tikzpicture}
    $
    \caption{The graph $\Delta(G_3)$ corresponding to $n=1$, $m=1$, and $k=2$}
    \label{fign1k2}
\end{figure}
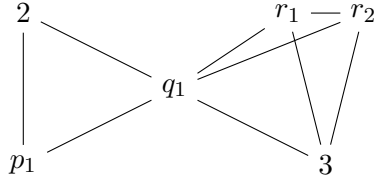
\end{example}

\begin{example}
Following Remark \ref{remPI}, we see that the graph in Figure \ref{fign1k2} can also be expressed as a direct product. This is a consequence of $k=2$ going above the bound of $2^{n+1}-3$. To show this, we again denote $A$ to be a finite solvable group such that $\Delta(A)$ is the singleton $K_1$. Furthermore, we denote $B$ as the finite solvable group whose prime character degree graph is the disconnected graph with complete components of sizes two and three. This graph was shown to occur in \cite{L5}. It is easy to see that $\Delta(G_3)$ in Figure \ref{fign1k2} is the same as the graph $\Delta(B\times A)$, which is shown in Figure \ref{figBA}. We again color-code the graph so that the blue represents the disconnected graph $\Delta(B)$, the red follows the singleton $\Delta(A)$, and the black represents the new edges built from the direct product, noting, once again, that we can guarantee that all primes considered are distinct. As before, the whole graph is $\Delta(B\times A)$.
\begin{figure}[htb]
    \centering
    $
\begin{tikzpicture}[scale=2]
\node (2) at (0,1) {\textcolor{blue}{$\bullet$}};
\node (p) at (0,0) {\textcolor{blue}{$\bullet$}};
\node (q) at (1,.5) {\textcolor{red}{$\bullet$}};
\node (r1) at (1.75,1) {\textcolor{blue}{$\bullet$}};
\node (r2) at (2.25,1) {\textcolor{blue}{$\bullet$}};
\node (3) at (2,0) {\textcolor{blue}{$\bullet$}};
\path[font=\small,>=angle 90,blue,thick]
(2) edge node [right] {$ $} (p)
(3) edge node [right] {$ $} (r1)
(3) edge node [right] {$ $} (r2)
(r1) edge node [right] {$ $} (r2);
\path[font=\small,>=angle 90,thick]
(q) edge node [right] {$ $} (2)
(q) edge node [right] {$ $} (p)
(q) edge node [right] {$ $} (3)
(q) edge node [right] {$ $} (r1)
(q) edge node [right] {$ $} (r2);
\end{tikzpicture}
    $
    \caption{The graph $\Delta(B\times A)$}
    \label{figBA}
\end{figure}
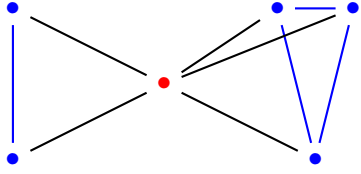
\end{example}

\begin{example}
Take $n=2$. Following the discussion in Remark \ref{remDP}, we see that $m\geq2^n-1=2^2-1=3$. For this example, we will take the minimum: $m=3$. Next, in regards to Remark \ref{remPI} for the bounds on $k$, we see that
$$
3=2^3-3-2=2^{n+1}-m-2\leq k\leq 2^{n+1}-3=2^3-3=5.
$$
Therefore, in order to avoid redundancies with direct products, there are only three options for $k$, which are $k=3$, $k=4$, or $k=5$. Regardless, we can write $\mathbf{\hat{p}}=p_1p_2$, and this then yields $\mathbf{\hat{q}}=2^{p_1p_2}-1=q_1q_2q_3$ and $\mathbf{\hat{r}}=\frac{2^{p_1p_2}+1}{3}=r_1r_2\cdots r_k$. Again employing \eqref{degrees}, we get that the set of character degrees is as follows:
\begin{align*}
\cd(G_{4|k})=\{1,&2,p_1,p_2,2p_1,2p_2,p_1p_2,2p_1p_2,3q_1q_2q_3r_1r_2\cdots r_k,2^{p_1p_2}q_1q_2q_3,\\
&2^{p_1p_2+1}q_1q_2q_3,2^{p_1p_2+1}p_1q_1q_2q_3,2^{p_1p_2+1}p_2q_1q_2q_3,2^{p_1p_2+1}p_1p_2q_1q_2q_3\}.
\end{align*}

The case of $k=3$ is attainable by taking $p_1=5$ and $p_2=17$. It is straightforward to get $q_1=31$, $q_2=131071$, and $q_3=9520972806333758431$. Moreover, one also gets $r_1=11$, $r_2=43691$, and $r_3=26831423036065352611$. One can verify that these are all indeed prime and that the corresponding prime character degree graph $\Delta(G_{4|3})$ is shown in Figure \ref{fig23}.
\begin{figure}[htb]
    \centering
    $
\begin{tikzpicture}[scale=1.5]
\node (p0) at (0,2) {$2$};
\node (p1) at (-.25,0) {$p_1$};
\node (p2) at (.25,0) {$p_2$};
\node (q1) at (2,1.25) {$q_1$};
\node (q2) at (1.75,.75) {$q_2$};
\node (q3) at (2.25,.75) {$q_3$};
\node (r1) at (3.75,2) {$r_1$};
\node (r2) at (4.25,2) {$r_2$};
\node (r3) at (4,1.5) {$r_3$};
\node (r0) at (4,0) {$3$};
\path[font=\small,>=angle 90]
(p0) edge node [right] {$ $} (p1)
(p0) edge node [right] {$ $} (p2)
(p1) edge node [right] {$ $} (p2)
(q1) edge node [right] {$ $} (q2)
(q1) edge node [right] {$ $} (q3)
(q2) edge node [right] {$ $} (q3)
(r0) edge node [right] {$ $} (r1)
(r0) edge node [right] {$ $} (r2)
(r0) edge node [right] {$ $} (r3)
(r1) edge node [right] {$ $} (r2)
(r1) edge node [right] {$ $} (r3)
(r2) edge node [right] {$ $} (r3)
(q1) edge node [right] {$ $} (p0)
(q1) edge node [right] {$ $} (p1)
(q1) edge node [right] {$ $} (p2)
(q1) edge node [right] {$ $} (r0)
(q1) edge node [right] {$ $} (r1)
(q1) edge node [right] {$ $} (r2)
(q1) edge node [right] {$ $} (r3)
(q2) edge node [right] {$ $} (p0)
(q2) edge node [right] {$ $} (p1)
(q2) edge node [right] {$ $} (p2)
(q2) edge node [right] {$ $} (r0)
(q2) edge node [right] {$ $} (r1)
(q2) edge node [right] {$ $} (r2)
(q2) edge node [right] {$ $} (r3)
(q3) edge node [right] {$ $} (p0)
(q3) edge node [right] {$ $} (p1)
(q3) edge node [right] {$ $} (p2)
(q3) edge node [right] {$ $} (r0)
(q3) edge node [right] {$ $} (r1)
(q3) edge node [right] {$ $} (r2)
(q3) edge node [right] {$ $} (r3);
\end{tikzpicture}
    $
    \caption{The graph $\Delta(G_{4|3})$ corresponding to $n=2$, $m=3$, and $k=3$}
    \label{fig23}
\end{figure}
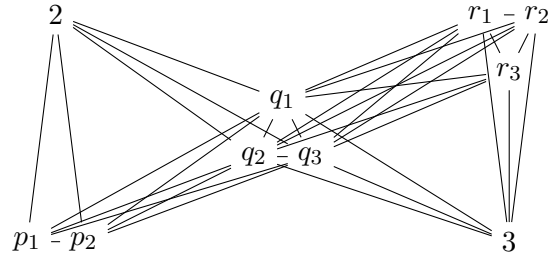

For $k=4$, one can, for example, take $p_1=7$ and $p_2=19$, which results in $q_1=127$, $q_2=524287$, and $q_3=163537220852725398851434325720959$. Moreover, one also gets $r_1=43$, $r_2=4523$, $r_3=174763$, and $r_4=106788290443848295284382097033$. This yields the graph $\Delta(G_{4|4})$ in Figure \ref{fig24}.
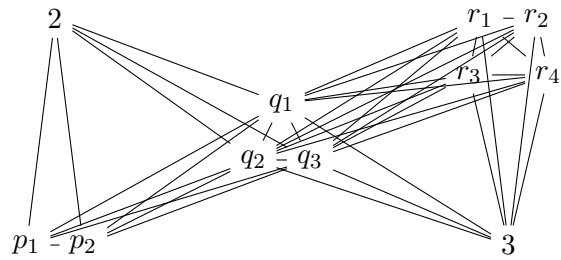
\begin{figure}[htb]
    \centering
    $
\begin{tikzpicture}[scale=1.5]
\node (p0) at (0,2) {$2$};
\node (p1) at (-.25,0) {$p_1$};
\node (p2) at (.25,0) {$p_2$};
\node (q1) at (2,1.25) {$q_1$};
\node (q2) at (1.75,.75) {$q_2$};
\node (q3) at (2.25,.75) {$q_3$};
\node (r1) at (3.75,2) {$r_1$};
\node (r2) at (4.25,2) {$r_2$};
\node (r3) at (3.65,1.5) {$r_3$};
\node (r4) at (4.35,1.5) {$r_4$};
\node (r0) at (4,0) {$3$};
\path[font=\small,>=angle 90]
(p0) edge node [right] {$ $} (p1)
(p0) edge node [right] {$ $} (p2)
(p1) edge node [right] {$ $} (p2)
(q1) edge node [right] {$ $} (q2)
(q1) edge node [right] {$ $} (q3)
(q2) edge node [right] {$ $} (q3)
(r0) edge node [right] {$ $} (r1)
(r0) edge node [right] {$ $} (r2)
(r0) edge node [right] {$ $} (r3)
(r0) edge node [right] {$ $} (r4)
(r1) edge node [right] {$ $} (r2)
(r1) edge node [right] {$ $} (r3)
(r1) edge node [right] {$ $} (r4)
(r2) edge node [right] {$ $} (r3)
(r2) edge node [right] {$ $} (r4)
(r3) edge node [right] {$ $} (r4)
(q1) edge node [right] {$ $} (p0)
(q1) edge node [right] {$ $} (p1)
(q1) edge node [right] {$ $} (p2)
(q1) edge node [right] {$ $} (r0)
(q1) edge node [right] {$ $} (r1)
(q1) edge node [right] {$ $} (r2)
(q1) edge node [right] {$ $} (r3)
(q1) edge node [right] {$ $} (r4)
(q2) edge node [right] {$ $} (p0)
(q2) edge node [right] {$ $} (p1)
(q2) edge node [right] {$ $} (p2)
(q2) edge node [right] {$ $} (r0)
(q2) edge node [right] {$ $} (r1)
(q2) edge node [right] {$ $} (r2)
(q2) edge node [right] {$ $} (r3)
(q2) edge node [right] {$ $} (r4)
(q3) edge node [right] {$ $} (p0)
(q3) edge node [right] {$ $} (p1)
(q3) edge node [right] {$ $} (p2)
(q3) edge node [right] {$ $} (r0)
(q3) edge node [right] {$ $} (r1)
(q3) edge node [right] {$ $} (r2)
(q3) edge node [right] {$ $} (r3)
(q3) edge node [right] {$ $} (r4);
\end{tikzpicture}
    $
    \caption{The graph $\Delta(G_{4|4})$ corresponding to $n=2$, $m=3$, and $k=4$}
    \label{fig24}
\end{figure}

Finally, for $k=5$, we pick $p_1=5$ and $p_2=13$. This will result in $q_1=31$, $q_2=8191$, and $q_3=145295143558111$, as well as $r_1=11$, $r_2=131$, $r_3=2731$, $r_4=409891$, and $r_5=7623851$. This gives us the prime character degree graph $\Delta(G_{4|5})$ in Figure \ref{fig25}.
\begin{figure}[htb]
    \centering
    $
\begin{tikzpicture}[scale=1.5]
\node (p0) at (0,2) {$2$};
\node (p1) at (-.25,0) {$p_1$};
\node (p2) at (.25,0) {$p_2$};
\node (q1) at (2,1.25) {$q_1$};
\node (q2) at (1.75,.75) {$q_2$};
\node (q3) at (2.25,.75) {$q_3$};
\node (r1) at (3.75,2) {$r_1$};
\node (r2) at (4.25,2) {$r_2$};
\node (r3) at (3.65,1.5) {$r_3$};
\node (r4) at (4.35,1.5) {$r_4$};
\node (r5) at (4,1.25) {$r_5$};
\node (r0) at (4,0) {$3$};
\path[font=\small,>=angle 90]
(p0) edge node [right] {$ $} (p1)
(p0) edge node [right] {$ $} (p2)
(p1) edge node [right] {$ $} (p2)
(q1) edge node [right] {$ $} (q2)
(q1) edge node [right] {$ $} (q3)
(q2) edge node [right] {$ $} (q3)
(r0) edge node [right] {$ $} (r1)
(r0) edge node [right] {$ $} (r2)
(r0) edge node [right] {$ $} (r3)
(r0) edge node [right] {$ $} (r4)
(r0) edge node [right] {$ $} (r5)
(r1) edge node [right] {$ $} (r2)
(r1) edge node [right] {$ $} (r3)
(r1) edge node [right] {$ $} (r4)
(r1) edge node [right] {$ $} (r5)
(r2) edge node [right] {$ $} (r3)
(r2) edge node [right] {$ $} (r4)
(r2) edge node [right] {$ $} (r5)
(r3) edge node [right] {$ $} (r4)
(r3) edge node [right] {$ $} (r5)
(r4) edge node [right] {$ $} (r5)
(q1) edge node [right] {$ $} (p0)
(q1) edge node [right] {$ $} (p1)
(q1) edge node [right] {$ $} (p2)
(q1) edge node [right] {$ $} (r0)
(q1) edge node [right] {$ $} (r1)
(q1) edge node [right] {$ $} (r2)
(q1) edge node [right] {$ $} (r3)
(q1) edge node [right] {$ $} (r4)
(q1) edge node [right] {$ $} (r5)
(q2) edge node [right] {$ $} (p0)
(q2) edge node [right] {$ $} (p1)
(q2) edge node [right] {$ $} (p2)
(q2) edge node [right] {$ $} (r0)
(q2) edge node [right] {$ $} (r1)
(q2) edge node [right] {$ $} (r2)
(q2) edge node [right] {$ $} (r3)
(q2) edge node [right] {$ $} (r4)
(q2) edge node [right] {$ $} (r5)
(q3) edge node [right] {$ $} (p0)
(q3) edge node [right] {$ $} (p1)
(q3) edge node [right] {$ $} (p2)
(q3) edge node [right] {$ $} (r0)
(q3) edge node [right] {$ $} (r1)
(q3) edge node [right] {$ $} (r2)
(q3) edge node [right] {$ $} (r3)
(q3) edge node [right] {$ $} (r4)
(q3) edge node [right] {$ $} (r5);
\end{tikzpicture}
    $
    \caption{The graph $\Delta(G_{4|5})$ corresponding to $n=2$, $m=3$, and $k=5$}
    \label{fig25}
\end{figure}
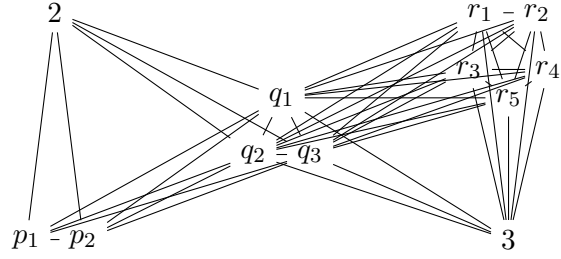
\end{example}

It is worth mentioning that by taking $n\geq3$, obtaining the minimum value of $m$ may be impossible. By Conjecture 4.3 on page 351 of \cite{DHS}, they hypothesize, in our context, that for $n\geq3$ we will have $m\geq2^n$.

\begin{example}
For $n=3$, notice that $m\geq2^n-1=2^3-1=7$. As per above, perhaps the minimum example is actually $m=8$. However, due to the computation becoming unruly as the product of primes becomes quite large, along with $\mathbf{\hat{p}}$ having neither $2$ nor $3$ as a divisor, the smallest $m$ we came across is $m=11$. This can be obtained by taking $p_1=5$, $p_2=7$, and $p_3=19$. Furthermore, under these primes, we then get $k=11$, which is permitted due to $k\leq2^{n+1}-3=2^4-3=13$.
\end{example}

We attain another example with $n=3$ by taking $p_1=5$, $p_2=13$, and $p_3=17$, which yield $m=12$ and $k=13$. One can also take $p_1=5$, $p_2=7$, and $p_3=17$, giving $m=14$ and $k=12$. As the number of primes increases, the number of edges makes the graph hard to follow. This is why we adopt the condensed version of the graph seen in Figure \ref{figpqr}.


\end{document}